\documentclass[12pt]{amsproc}

\usepackage{amsmath, amsthm, amssymb, amsfonts, mathrsfs, amscd}

\def\AA{{\mathbb A}}

\def\FF{{\mathbb F}}

\def\PP{{\mathbb P}}

\def\ZZ{{\mathbb Z}}

\def\0{{\mathbf 0}}
\def\1{{\mathbf 1}}
\def\a{{\mathbf a}}

\def\x{{\mathbf x}}

\def\Ocal{{\mathcal O}}

\def\Xcal{{\mathcal X}}

\def\pfrak{{\mathfrak p}}

\def\Kbar{{\bar K}}

\def\PGL{\mathrm{PGL}}
\def\GL{\mathrm{GL}}

\def\ord{\mathrm{ord}}

\def\Res{\mathrm{Res}}

\def\Aff{\mathrm{Aff}}

\def\rk{\mathrm{rank}}

\theoremstyle{plain}

\newtheorem{thm}{Theorem}

\newtheorem{cor}[thm]{Corollary}
\newtheorem{prop}[thm]{Proposition}
\newtheorem{lem}[thm]{Lemma}

\theoremstyle{definition}
\newtheorem*{dfn}{Definition}

\begin{document}

\title[Global Minimal Models]{Global minimal models for endomorphisms of projective space}

\author{Clayton Petsche and Brian Stout}

\address{Clayton Petsche; Department of Mathematics; Oregon State University; Corvallis OR 97331 U.S.A.}

\email{petschec@math.oregonstate.edu}

\address{Brian Stout; Ph.D. Program in Mathematics; CUNY Graduate Center; 365 Fifth Avenue; New York, NY 10016-4309 U.S.A.}

\email{bstout@gc.cuny.edu}

\thanks{Date of last revision: March 9 2013.}

\begin{abstract}
We prove the existence of global minimal models for rational morphisms $\phi:\PP^N\rightarrow\PP^N$ of projective space defined over the field of fractions of a principal ideal domain.
\end{abstract}

\maketitle


\section{Definitions and statement of the main results}\label{Introduction}

Let $R$ be a principal ideal domain (PID) with field of fractions $K$, and let $N$ be a positive integer.  In this paper, our primary objects of study are morphisms $\phi:\PP^N\to\PP^N$ defined over $K$.  Fixing a choice of homogeneous coordinates $\x=(x_0,\dots,x_N)$ on $\PP^N$, we may write $\phi$ explicitly as 
\begin{equation}\label{HomogMap}
\phi(x_0:\dots:x_N)=(\Phi_0(x_0,\dots,x_N):\dots:\Phi_N(x_0,\dots,x_N)),
\end{equation}
where $\Phi:\AA^{N+1}\to \AA^{N+1}$ is a map defined by an $(N+1)$-tuple $\Phi=(\Phi_0,\dots,\Phi_N)$ of forms of some common degree $d\geq1$ in the variables $x_0,x_1,\dots,x_N$, with the property that 
\begin{equation}\label{NonVanishing}
\Phi(\a)\neq\0\text{ whenever }\a\in\AA^{N+1}(\Kbar)\setminus\0,
\end{equation}
or equivalently that
\begin{equation}\label{ResultantNonVanishing}
\Res(\Phi)\neq0,
\end{equation}
where $\Res(\Phi)$ is the resultant of $\Phi$, a certain homogeneous integral polynomial in the coefficients of the forms $\Phi_n$; see Proposition~\ref{ResProp} for a review of the necessary facts about the resultant.  We refer to $d$ as the {\em algebraic degree} of $\phi$, and we refer to the map $\Phi$, which is uniquely determined by $\phi$ up to multiplication by a nonzero scalar in $K$, as a {\em homogeneous lift} for $\phi$.  

Conversely, starting with any map $\Phi:\AA^{N+1}\to \AA^{N+1}$ defined by an $(N+1)$-tuple $\Phi=(\Phi_0,\dots,\Phi_N)$ of forms of some common degree $d\geq0$, such that $\Phi$ satisfies the nonvanishing condition $(\ref{NonVanishing})$, the formula $(\ref{HomogMap})$ gives rise to a morphism $\phi:\PP^N\to\PP^N$ of algebraic degree $d$.  

In the study of the dynamical system obtained from iteration of the morphism $\phi$, it is generally true that the dynamical properties of $\phi$ are left unchanged when it is replaced with its conjugate $f\circ\phi\circ f^{-1}$ by an element $f$ of the automorphism group $\PGL_{N+1}(K)$ of $\PP^N$ over $K$.  Given a representative $A\in\GL_{N+1}(K)$ for $f$ under the quotient map $\GL_{N+1}\to\PGL_{N+1}$, and given a homgoeneous lift $\Phi:\AA^{N+1}\to \AA^{N+1}$ for $\phi$, observe that the map $\Psi=A\circ\Phi\circ A^{-1}:\AA^{N+1}\to \AA^{N+1}$ is a homogeneous lift for $\psi=f \circ\phi\circ f^{-1}$.  It is therefore natural to offer the following loosening of the notion of a homogeneous lift for $\phi$.

\begin{dfn}
Let $\phi:\PP^N\to\PP^N$ be a morphism defined over $K$.  A {\em model} for $\phi$ over $K$ is a map $\Psi:\AA^{N+1}\to \AA^{N+1}$ given by $\Psi=A\circ\Phi\circ A^{-1}$ for some homogeneous lift $\Phi:\AA^{N+1}\to \AA^{N+1}$ of $\phi$ and some linear automorphism $A\in\GL_{N+1}(K)$ of $\AA^{N+1}$.
\end{dfn}

While $\PGL_{N+1}(K)$-conjugation does not affect purely dynamical properties of morphisms, it does have subtle and unpredictable effects on integrality and divisibility properties in the ring $R$.    For each nonzero prime ideal $\pfrak$ of $R$, denote by $K_\pfrak$ the completion of $K$ with respect to the $\pfrak$-adic valuation, and let $R_\pfrak$ be the subring of $\pfrak$-integral elements of $K_\pfrak$.  Let $\FF_\pfrak=R_\pfrak/\pfrak R_\pfrak$ be the residue field at $\pfrak$, and denote by $x\mapsto\tilde{x}_\pfrak$ the surjective reduction map $R_\pfrak\to\FF_\pfrak$.

Given a model $\Psi:\AA^{N+1}\to \AA^{N+1}$ for a morphism $\phi:\PP^N\to\PP^N$ defined over $K_\pfrak$, we declare that $\Psi$ is {\em integral} (or {\em $\pfrak$-integral}) if each form $\Psi_n$ has coefficients in $R_\pfrak$.  If $\Psi$ is $\pfrak$-integral, then we may reduce the coefficients modulo $\pfrak$ and obtain a homogeneous map $\tilde{\Psi}_\pfrak:\AA^{N+1}\to\AA^{N+1}$ defined over the residue field $\FF_\pfrak$.

\begin{dfn}
A morphism $\phi:\PP^N\to\PP^N$ defined over $K_\pfrak$ has {\em good reduction} if $\phi$ has a $\pfrak$-integral model $\Psi:\AA^{N+1}\to \AA^{N+1}$ satisfying either (and therefore both) of the following two equivalent conditions:
\begin{itemize}
\item[{\bf (a)}]  the reduced map $\tilde{\Psi}_\pfrak:\AA^{N+1}\to\AA^{N+1}$ satisfies $\tilde{\Psi}_\pfrak(\a)\neq\0$ whenever $\a\in\AA^{N+1}(\overline{\FF}_\pfrak)\setminus\0$;
\item[{\bf (b)}]  $\Res(\Psi)\in R_\pfrak^\times$.
\end{itemize}
\end{dfn}

According to condition {\bf (a)}, this definition has the following fairly intuitive interpretation: a morphism $\phi:\PP^N\to\PP^N$ of algebraic degree $d\geq1$ defined over $K_\pfrak$ has good reduction precisely when it is $\PGL_{N+1}(K)$-conjugate to a morphism $\psi:\PP^N\to\PP^N$ for which reduction modulo $\pfrak$ gives rise to a morphism $\tilde{\psi}_{\pfrak}:\PP^N\to\PP^N$ of algebraic degree $d$ defined over the residue field $\FF_\pfrak$.  The equivalence of conditions {\bf (a)} and {\bf (b)} is a simple conseuquence of basic properties of the resultant, along with the fact that the unit group $R_\pfrak^\times$ is precisely the set of elements in $R_\pfrak$ whose image is nonzero under the reduction map $R_\pfrak\to\FF_\pfrak$.

If $\Psi:\AA^{N+1}\to \AA^{N+1}$ is an arbitrary $\pfrak$-integral model for $\phi$, then $\ord_\pfrak(\Res(\Psi))\geq0$ since $\Res(\Psi)$ is an integral polynomial in the coefficients of $\Psi$; good reduction at $\pfrak$ occurs precisely when a $\pfrak$-integral model $\Psi$ can be found with $\ord_\pfrak(\Res(\Psi))=0$.  Even in the case of bad reduction, however, one might still ask for a $\pfrak$-integral model $\Psi$ for $\phi$ with $\ord_\pfrak(\Res(\Psi))$ as small as possible.

\begin{dfn}
Let $\phi:\PP^N\to\PP^N$ be a morphism defined over $K_\pfrak$.  A $\pfrak$-integral model $\Psi:\AA^{N+1}\to \AA^{N+1}$ for $\phi$ is {\em minimal} (or {\em $\pfrak$-minimal}) if $\ord_\pfrak(\Res(\Psi))$ is minimal among all $\pfrak$-integral models $\Psi$ for $\phi$.
\end{dfn}

We can now state the main theorem of this paper.   Given a morphism $\phi:\PP^N\to\PP^N$ defined over $K$, and a nonzero prime ideal $\pfrak$ of $R$, there always exists a minimal $\pfrak$-integral model $\Psi$ for $\phi$: start with an arbitrary model defined over $K_\pfrak$, scale by a $\pfrak$-adic uniformizing parameter to obtain a $\pfrak$-integral model $\Psi$, and among all such $\Psi$, select one for which $\ord_\pfrak(\Res(\Psi))$ is minimal.  A priori these minimal $\pfrak$-integral models vary from prime to prime, but it is natural to ask whether one can find a {\em global minimal model}; that is, a model defined over $R$ which is simultaneously a minimal $\pfrak$-integral model at all prime ideals $\pfrak$ of $R$.  

\begin{thm}\label{GMMTheoremIntro}
Let $R$ be a PID with field of fractions $K$, and let $\phi:\PP^N\to\PP^N$ be a morphism defined over $K$.  Then $\phi$ has a model $\Psi:\AA^{N+1}\to \AA^{N+1}$, with coefficients in $R$, and which is $\pfrak$-minimal for all nonzero prime ideals $\pfrak$ of $R$.
\end{thm}

An interesting special case of Theorem~\ref{GMMTheoremIntro} occurs when the morphism $\phi:\PP^N\to\PP^N$ is assumed to have {\em everywhere} good reduction; that is, when $\phi$ has good reduction at all nonzero prime ideals $\pfrak$ of $R$.  While this represents an extremal case of Theorem~\ref{GMMTheoremIntro}, it is perhaps not as special as it may appear: since any morphism $\phi:\PP^N\to\PP^N$ defined over $K$ has good reduction at all except a finite set $S$ of nonzero prime ideals $\pfrak$ of $R$, replacing $R$ with the larger PID $R_S=\{r\in K\mid \ord_\pfrak(r)\geq0\text{ for all  }\pfrak\not\in S\}$, we observe that $\phi$ has everywhere good reduction over $R_S$.

\begin{cor}\label{EGRTheoremIntro}
Let $R$ be a PID with field of fractions $K$, let $\phi:\PP^N\to\PP^N$ be a morphism defined over $K$, and assume that $\phi$ has good reduction at all nonzero prime ideals $\pfrak$ of $R$.  Then $\phi$ has a model $\Psi:\AA^{N+1}\to \AA^{N+1}$, with coefficients in $R$, such that $\Res(\Psi)\in R^\times$.
\end{cor}

In the case $N=1$, Theorem~\ref{GMMTheoremIntro} was proposed by Silverman (\cite{MR2316407} pp. 236-237) and proved by Bruin-Molnar \cite{BruinMolnar}; thus our result generalizes this to arbitrary dimension $N\geq1$.  Our proof is not a straightforward generalization the proof by Bruin-Molnar, however.  In \cite{BruinMolnar}, it is shown that, in order to produce a global minimal model for a rational map $\phi:\PP^1\to\PP^1$, one only needs to consider conjugates $f\circ\phi\circ f^{-1}$ of $\phi$ by $f$ in the group $\Aff_2$ of automorphisms leaving $\infty$ fixed; i.e. automorphisms taking the form $f(x)=\alpha x+\beta$ in an affine coordinate $x$.  We do not know whether, in the higher dimensional case, a generalization of $\Aff_2$ can be used in a similar fashion leading to a proof of Theorem~\ref{GMMTheoremIntro}.

Our proof of Theorem~\ref{GMMTheoremIntro} relies on the theory of lattices over a PID, and in particular on the action of the adelic general linear group $\GL_{n}(\AA_R)$ on the space of all such lattices of rank $n$.  The main technical lemma of this paper is a factorization of the group $\GL_{n}(\AA_R)$ as the product of the subgroup $\GL_{n}(K)$ of principal adeles with the direct product $\GL_{n}^0(\AA_R)=\prod_\pfrak\GL_n(R_\pfrak)$.  When $R$ is a ring of $S$-integers in a number field $K$, this follows from a more general result of Borel \cite{MR0202718} on the finiteness of the class number of $\GL_n$.  Since we have not been able to find the required material worked out over an arbitrary PID, in this paper we give a self-contained treatment.  

Theorem~\ref{GMMTheoremIntro} and Corollary~\ref{EGRTheoremIntro} may find arithmetic applications in the setting of a global field $K$ (a number field or a function field with a finite constant field) and a finite subset $S$ of places of $K$.  After possibly replacing $S$ with a suitable larger finite set of places, it is always possible to obtain the situation in which the ring $\Ocal_{S}$ of $S$-integers is a PID.  In \cite{Stout}, the first author uses Theorem~\ref{GMMTheoremIntro} to prove a finiteness theorem for twists of rational maps having prescribed good reduction.  Other applications of this idea, in slightly different contexts, can be found in the proof of Shafarevich's Theorem for elliptic curves (see \cite{MR2514094} $\S$IX.6), as well as an analogue for rational maps due to Petsche \cite{Petsche}.  

The first author's research was supported in part by grant DMS-0901147 of the National Science Foundation.  The second author is supported by grant DMS-0739346 of the National Science Foundation.


\section{Global and local lattices over a PID}\label{LatticeSect}

Throughout this paper $R$ is a PID with field of fractions $K$, and $R^\times$ denotes the group of units in $R$.  The set of non-zero prime (and thus maximal) ideals of $R$ will be denoted by $M_R$.  For each $\pfrak\in M_R$, let $K_\pfrak$ be the completion of $K$ with respect to the discrete valuation $\ord_\pfrak(\cdot)$ on $K$, and let 
\begin{equation*}
\begin{split}
R_\pfrak & = \{a\in K_\pfrak\mid \ord_\pfrak(a)\geq0\} \\
R_\pfrak^\times & = \{a\in K_\pfrak\mid \ord_\pfrak(a)=0\} 
\end{split}
\end{equation*}
be the subring of $\pfrak$-integral elements of $K_\pfrak$, and its unit group, respectively.  It is a standard exercise to check the identities 
\begin{equation}\label{IntegralLocalGlobal}
\begin{split}
R & =\{a\in K\mid \ord_\pfrak(a)\geq0\text{ for all } \pfrak\in M_R\} \\
R^\times & =\{a\in K\mid \ord_\pfrak(a)=0\text{ for all } \pfrak\in M_R\}.
\end{split}
\end{equation}

\begin{prop}\label{LatticeProp}
Let $X$ be an $R$-submodule of $K^n$.  Then the following three conditions are equivalent:
\begin{enumerate}
\item[(i)] $X$ is free and $\text{rank}(X)=n$.
\item[(ii)] $a R^n\subseteq X \subseteq b R^n$ for some $a,b\in K^\times$.
\item[(iii)] $X=AR^n$ for some $A\in \GL_{n}(K)$.
\end{enumerate}
\end{prop}

\begin{proof}
(i) $\Rightarrow$ (iii): If (i) holds, let $A$ be an $n\times n$ matrix over $K$ whose columns form an $R$-basis for $X$. Then $X=AR^n$ and $A$ is nonsingular, hence $A\in\GL_{n}(K)$.  (If $A$ were singular, then there would be a non-trivial $K$-linear dependence among the columns of $A$; multiplying by the product of the denominators of the coefficients of this linear dependence, we would obtain a linear dependence with coefficients in $R$, in violation of the assumption that the columns of $A$ form an $R$-basis for $X$.)

(iii) $\Rightarrow$ (ii): If (iii) holds, let $A\in\GL_{n}(K)$ such that $X=AR^n$. Let $a_{ij}$ denote the entries of $A$ and let $b$ be the reciprocal of the product of the denominators of  the $a_{ij}$ for $1\leq i,j\leq n$.  Then $b^{-1}X=b^{-1}AR^n\subseteq R^n$ since $b^{-1}A$ has entries in $R$, and therefore $X\subseteq bR^n$.  Let $b_{ij}$ denote the entries of $A^{-1}$ and let $a$ be the product of the denominators of the $b_{ij}$ for $1\leq i,j \leq n$. Then $aR^n\subseteq aA^{-1}X\subseteq X$ since $aA^{-1}$ has entries in $R$ and $X$ is an $R$-module.

(ii) $\Rightarrow$ (i): Since $x\mapsto a x$ is an isomorphism $R^n\to a R^n$, we see that $a R^n$ is a free $R$-module of rank $n$; the same is true of $b R^n$.  Since $R$ is a PID, it follows from Theorem 7.1 of \cite{MR1878556} that any $R$-submodule of $b R^n$ is also free of rank less than or equal to $n$.  Since $X\subseteq b R^n$, $X$ is free and $\rk(X)\leq\rk(b R^n)$. The inequality $\rk(a  R^n)\leq\rk(X)$ now follows from the same theorem, as $X$ has been shown to be free.  Because $a R^n$ and $b  R^n$ are both of rank $n$, it follows that $X$ has rank $n$.
\end{proof}

\begin{dfn}
An $R$-\emph{lattice} in $K^n$ is a free $R$-submodule of $K^n$ of rank $n$.  
\end{dfn}

For each $\pfrak\in M_R$, the local ring $R_\pfrak$ is itself a PID, and thus Proposition~\ref{LatticeProp} applies to $R_\pfrak$-submodules of $K_\pfrak^n$.  In particular, an $R_\pfrak$-lattice in $K_\pfrak^n$ is a free $R_\pfrak$-submodule of $K_\pfrak^n$ of rank $n$.  

If $X$ is an $R$-lattice in $K^n$ and $\pfrak\in M_R$ is a nonzero prime ideal of $R$, there is a natural way to associate to $X$ an $R_\pfrak$-lattice $X_\pfrak$ in $K^n_\pfrak$.  By Proposition 3, we may find some $A\in\GL_{n}(K)$ such that $X=AR^n$, and we define $X_\pfrak=AR^n_\pfrak$.  This definition does not depend on the choice of matrix $A$.  For if $X=BR^n$, then $A^{-1}B$ is an isomorphism $R^n\rightarrow R^n$, and therefore $A^{-1}B\in\GL_{n}(R)\subseteq \GL_{n}(R_\pfrak)$.  Then $A^{-1}BR^n_\pfrak=R^n_\pfrak$ and therefore $BR^n_\pfrak=AR^n_\pfrak$.  The definition of $X_\pfrak$ is equivalent to the $R_\pfrak$-module $X\otimes_R R_\pfrak$ obtained by extension of scalars.

\begin{lem}
Let $X$ be an $R$-lattice in $K^n$.  Then for every $\pfrak\in M_R$, $X_\pfrak$ is an $R_\pfrak$-lattice in $K_\pfrak^n$, and for almost every $\pfrak\in M_R$, $X_\pfrak=R_\pfrak^n$. 
\end{lem}

\begin{proof}
Let $X$=$AR^n$ for $A\in\GL_{n}(K)$. For any $\pfrak\in M_R$, we have that $X_\pfrak=AR^n_\pfrak$ and therefore $X_\pfrak$ is an $R_\pfrak$-lattice in $K_\pfrak^n$ by Proposition~\ref{LatticeProp}.  Furthermore, $X_\pfrak=R^n_\pfrak$ for all $\pfrak\in M_R$ except for the finitely many $\pfrak$ for which $A\not\in\GL_{n}(R_\pfrak)$.  These primes correspond to the irreducible elements which occur in the denominators of the entries of $A$ or in the numerator of the determinant of $A$.
\end{proof}

\begin{lem}
Conversely, suppose that $(X_\pfrak)$ is a collection of $R_\pfrak$-lattices in $K_\pfrak^n$ for each $\pfrak\in M_R$, such that $X_\pfrak=R_\pfrak^n$ for almost every $\pfrak$.  Then 
\begin{equation*}
X'=\lbrace x\in K^n | x\in X_\pfrak\text{ for all }\pfrak\rbrace
\end{equation*}
is an $R$-lattice in $K^n$, and $X'_\pfrak=X_\pfrak$ for each prime $\pfrak\in M_R$.
\end{lem}

\begin{proof}
$X'$ is plainly an $R$-submodule of $K^n$ because $R\subseteq R_\pfrak$ for all $\pfrak\in M_R$ and each $X_\pfrak$ is an $R_\pfrak$-submodule of $K_\pfrak^n$.  By Proposition 3, to show that $X'$ is free of rank $n$ it is sufficient to show that $aR^n\subseteq X'\subseteq bR^n$ for some $a,b\in K^\times$.  As each $X_\pfrak$ is an $R_\pfrak$-lattice in $K^n_\pfrak$, we know by Proposition 3 that a similar chain of inclusions $a_\pfrak R^n_\pfrak\subseteq X_\pfrak\subseteq b_\pfrak R^n_\pfrak$ holds for each each prime $\pfrak$ where $a_\pfrak,b_\pfrak\in K^\times_\pfrak$. By the assumption $X_\pfrak=R_\pfrak^n$ for almost every $\pfrak$, we may assume that $a_\pfrak=b_\pfrak=1$ for almost every $\pfrak$.  Because $R$ is a PID we may assume that both $a_\pfrak$ and $b_\pfrak$ are powers of $\pfrak$-adic uniformizing parameters in $R$.  Let $a=\Pi_\pfrak a_\pfrak, b=\Pi_\pfrak b_\pfrak\in K^\times$ and it follows that $aR^n_\pfrak\subseteq X_\pfrak\subseteq bR^n_\pfrak$. Using (\ref{IntegralLocalGlobal}) we have that $aR^n=\lbrace x\in K^n | x\in aR^n_\pfrak\text{ for all } \pfrak \rbrace$ and that $bR^n=\lbrace x\in K^n | x\in bR^n_\pfrak\text{ for all } \pfrak \rbrace$. Therefore $aR^n\subseteq X'\subseteq bR^n$ and we conclude $X'$ to be an $R$-lattice.

Lastly, we show that $X'_\pfrak=X_\pfrak$ for all $\pfrak\in M_R$.  The inclusion $X'_\pfrak\subseteq X_\pfrak$ follows immediately from the definitions: Proposition~\ref{LatticeProp} provides an element $A\in\GL_n(K)$ such that $X'=AR^n$, and $X'_\pfrak=AR_\pfrak^n$.  Since $X'\subseteq X_\pfrak$, the column vectors of $A$ are in $X_\pfrak$, whereby $X'_\pfrak=AR^n_\pfrak\subseteq X_\pfrak$.

To show equality $X'_\pfrak=X_\pfrak$ for all $\pfrak\in M_R$, suppose there exists some $\pfrak_0\in M_R$ with proper inclusion $X'_{\pfrak_0}\subsetneq X_{\pfrak_0}$; we will derive a contradiction.

Let $\AA^n_R$ be the affine adelic space over $R$.  This space is the restricted direct product of the affine spaces $K^n_\pfrak$ with respect to the subsets $R^n_\pfrak$.  Specifically,
\begin{equation*}
\AA^n_R = \bigg\{(a_\pfrak)\in\prod_{\pfrak\in M_R}K_\pfrak^n\bigg|  \,\,
		\begin{tabular}{@{}l@{}}
           $a_\pfrak\in R_\pfrak$ for almost all $\pfrak$ 
      \end{tabular}\bigg\}.
\end{equation*}
Thus an arbitrary element of $\AA^n_R$ is a tuple $(a_\pfrak)$, indexed by the primes $\pfrak\in M_R$, where each $a_\pfrak\in K^n_\pfrak$, and where $a_\pfrak\in R^n_\pfrak$ for almost all $\pfrak$.  The affine adelic space has a topology whose basis consists of sets of the form $\Pi_\pfrak U_\pfrak$, where each $U_\pfrak$ is an open subset of $K^n_\pfrak$ and where $U_\pfrak=R^n_\pfrak$ for almost all $\pfrak$.  Naturally, $K^n$ is a subset of $\AA^n_R$ by identifying $a\in K^n$ with the principal adele $(a_\pfrak)$, where $a_\pfrak=a$ for all $\pfrak$.

Define subsets of $\AA_R^n$ by $Y'=\Pi_\pfrak X'_\pfrak$ and $Y=\Pi_\pfrak X_\pfrak$.  Since we have already shown that $X'_\pfrak\subseteq X_\pfrak$ for all $\pfrak\in M_R$, and since we have assumed that $X'_{\pfrak_0}\subsetneq X_{\pfrak_0}$ for some $\pfrak_0\in M_R$, it follows that $Y'\subsetneq Y$.  Since an arbitrary $R_\pfrak$-lattice is both open and closed in $K_\pfrak^n$, it follows from the definition of the restricted direct product topology that $Y$ and $Y'$ are both open and closed in $\AA^n_R$, and therefore that $Y\setminus Y'$ is a nonempty open subset of $\AA^n_R$. It follows from a standard argument that $K^n$ is a dense subset of $\AA^n_R$. (When $R=\ZZ$, this is the most basic form of the weak approximation theorem, the proof of which can be found in Cassels (\cite{MR911121} Ch. II, $\S$14, 15); a direct generalization of this argument holds for an arbitrary PID.)  Therefore, there exists $x\in K^n$ whose principal adele $(x)$ is an element of $Y\setminus Y'$.  Since $(x)\in Y=\Pi_\pfrak X_\pfrak$, we have $x\in X_\pfrak$ for all $\pfrak$ and hence by definition, $x\in X'$.  It follows that $x\in X'_\pfrak$ for all $\pfrak$ and consequently $(x)\in \Pi_\pfrak X'_\pfrak=Y'$.  This contradiction implies that $X'_\pfrak=X_\pfrak$ for all  $\pfrak\in M_R$.
\end{proof}


\section{The adelic general linear group over a PID}\label{FactorizationSect}

The adelic general linear group $\GL_{n}(\AA_R)$ associated to $R$ is the restricted direct product of the groups $\GL_{n}(K_\pfrak)$ with respect to the subgroups $\GL_{n}(R_\pfrak)$.  More specifically,
\begin{equation*}
\GL_{n}(\AA_R) = \bigg\{(A_\pfrak)\in\prod_{\pfrak\in M_R}\GL_{n}(K_\pfrak)\bigg|  \,\,
		\begin{tabular}{@{}l@{}}
           $A_\pfrak\in \GL_{n}(R_\pfrak)$ for almost all $\pfrak$ 
      \end{tabular}\bigg\}.
\end{equation*}

The main result of this section shows that the group $\GL_{n}(\AA_R)$ factors into a product of two natural subgroups.  First, $\GL_{n}(K)$ embeds into $\GL_{n}(\AA_R)$ by the identification of each $A\in \GL_{n}(K)$ with the its {\em principal} adele $(A_\pfrak)$, defined by $A_\pfrak = A$ for all $\pfrak\in M_R$.  The second subgroup of $\GL_{n}(\AA_R)$ is 
\begin{equation*}
\GL_{n}^0(\AA_R) = \prod_{\pfrak\in M_R}\GL_{n}(R_\pfrak),
\end{equation*}
the direct product of the $R_\pfrak$-integral subgroups $\GL_{n}(R_\pfrak)$, over all primes $\pfrak\in M_R$.

\begin{prop}\label{MainGroupThm}
$\GL_{n}(\AA_R) = \GL_{n}^0(\AA_R)\GL_{n}(K)$.
\end{prop}

The following lemma contains most of work toward the proof of Proposition~\ref{MainGroupThm}.

\begin{lem}
Let $\Xcal_R$ denote the set of $R$-lattices in $K^n$.  There exists a transitive group action
\begin{equation*}
\begin{split}
\GL_{n}(\AA_R)\times\Xcal_R & \rightarrow\Xcal_R \\
(A,X) & \mapsto A\cdot X,
\end{split}
\end{equation*}
where $A\cdot$X is defined to be the $R$-lattice
\begin{equation*}
A\cdot X =\lbrace x\in K^n|x\in A_\pfrak X_\pfrak\text{ for all }\pfrak\rbrace.
\end{equation*}
Moreover, the stabilizer in $\GL_n(\AA_R)$ of the trivial lattice $R^n$ is $\GL_{n}^0(\AA_R)$.
\end{lem}

\begin{proof}
Let $A,B\in\GL_{n}(\AA_R)$ and $X\in\Xcal_R$.  The fact that $A\cdot X$ is an $R$-lattice in $K^n$ follows from Lemma 5.

Let $I=(I_\pfrak)$ denote the identity adele: $I_\pfrak$ is the identity matrix in $\GL_{n}(K_\pfrak)$ for each $\pfrak\in M_R$.  We show that $I\cdot X=X$, or equivalently, that 
\begin{equation*}
\lbrace x\in K^n | x\in X_\pfrak\text{ for all } \pfrak\rbrace = X.
\end{equation*}
First, if $X=R^n$ then the desired identity 
\begin{equation*}
\lbrace x\in K^n| x\in R^n_\pfrak\text{ for all }\pfrak\rbrace = R^n
\end{equation*}
follows immediately from $(\ref{IntegralLocalGlobal})$, and thus $I\cdot R^n=R^n$.  Now let $X$ be arbitrary.  By Proposition 3, $X=AR^n$ for some $A\in\GL_{n}(K)$, and by definition $X_\pfrak=AR^n_\pfrak$.  It follows that
\begin{equation*}
\begin{split}
I\cdot X & = \lbrace x\in K^n|x\in X_\pfrak=AR^n_\pfrak\text{ for all }\pfrak\rbrace \\
	& = \lbrace Ax | x\in K^n, x\in R^n_\pfrak\text{ for all } \pfrak\rbrace \\
	& = AR^n=X.
\end{split}
\end{equation*}

The equality $A\cdot(B\cdot X)=(AB)\cdot X$ follows from the identity $(B\cdot X)_\pfrak=B_\pfrak X_\pfrak$, which itself is a trivial consequence of Lemma 5.  Specifically, 
\begin{equation*}\begin{array}{lcl}
A\cdot(B\cdot X)&=&\lbrace x\in K^n|x\in A_\pfrak(B\cdot X)_\pfrak\text{ for all }\pfrak\rbrace\\
&=&\lbrace x\in K^n|x\in A_\pfrak(B_\pfrak X_\pfrak)\text{ for all }\pfrak\rbrace\\
&=&\lbrace x\in K^n|x\in (AB)_\pfrak X_\pfrak\text{ for all }\pfrak\rbrace\\
&=&(AB)\cdot X.
\end{array}\end{equation*}

The transitivity of the action follows from Proposition 3: for any lattice $X$ there is $A\in\GL_{n}(K)$ such that $X=AR^n$ and considering $A$ as a principal adele it then follows that $X=A\cdot R^n$. Therefore every $R$-lattice in $K^n$ is in the $\GL_{n}(\AA_R)$-orbit of the trivial lattice.

Finally, we must show that the stabilizer in $\GL_n(\AA_R)$ of the trivial lattice $R^n$ is $\GL_{n}^0(\AA_R)$; in other words, that
\begin{equation*}
\lbrace A\in \GL_{n}(\AA_R)| A\cdot R^n=R^n\rbrace = \GL_{n}^0(\AA_R).
\end{equation*}
If $A=(A_\pfrak)\in \GL_{n}^0(\AA_R)$, then $A_\pfrak\in\GL_n(R_\pfrak)$ for all $\pfrak\in M_R$, which implies that $A_\pfrak R^n_\pfrak=R^n_\pfrak$.  We conclude using $(\ref{IntegralLocalGlobal})$ that 
\begin{equation*}
\begin{split}
A\cdot R^n & = \lbrace x\in K^n|x\in A_\pfrak R^n_\pfrak \text{ for all }\pfrak\rbrace \\
	& = \lbrace x\in K^n|x\in R^n_\pfrak \text{ for all }\pfrak\rbrace \\
	& = R^n.
\end{split}
\end{equation*}

Conversely, suppose $A=(A_\pfrak)\in \GL_{n}(\AA_R)$ such that $A\cdot R^n=R^n$, which, by definition means that
\begin{equation}\label{LeftAndRight}
\lbrace x\in K^n | x\in A_\pfrak R^n_\pfrak\text{ for all }\pfrak\rbrace = R^n.
\end{equation}
Let $X$ and $Y$ denote the left-hand side and right-hand side of $(\ref{LeftAndRight})$, respectively, and fix $\pfrak\in M_R$.  Then trivially $Y_\pfrak=R_\pfrak^n$, and Lemma 5 shows that $X_\pfrak=A_\pfrak R^n_\pfrak$.  We conclude that $A_\pfrak R^n_\pfrak=R^n_\pfrak$, and this implies that $A\in\GL_n(R_\pfrak)$.  [Proof: Let $\{e_i\}\in R^n_\pfrak$ be the standard basis.  Then $A_\pfrak e_i\in R^n_\pfrak$ is the $i^{th}$ column of $M$, showing that $A_\pfrak$ has coefficients in $R_\pfrak$. Similarly, $A_\pfrak^{-1}$ fixes $R_\pfrak^n$ and therefore $A_\pfrak^{-1}$ has coefficients in $R_\pfrak$].  Hence $A_\pfrak\in\GL_{n}(R_\pfrak)$ for every prime $\pfrak$, and so by definition $A\in\GL_{n}^0(\AA_r)$.
\end{proof}

\begin{proof}[Proof of Proposition~\ref{MainGroupThm}]
Let $A\in \GL_{n}(\AA_R)$ be an arbitrary adele. Let $X=A^{-1}\cdot R^n$ be the lattice obtained by letting $A^{-1}$ act on the trivial lattice. By Proposition~\ref{LatticeProp}, $X=BR^n$ for $B\in\GL_{n}(K)$. Both $A^{-1}$ and $B$ take $R^n$ bijectively onto $X$, so $AB$ fixes $R^n$ and therefore lies in the stabilizer $\GL_{n}^0(\AA_R)$, say $AB=C$ for $C\in \GL_{n}^0(\AA_R)$.  Therefore $A=CB^{-1}\in\GL_{n}^0(\AA_R)\GL_{n}(K)$.
\end{proof}


\section{The existence of global minimal models}

In this section we prove the main results of the paper, Theorem~\ref{GMMTheoremIntro} and Corollary~\ref{EGRTheoremIntro}.  First, however, we give a proposition summarizing the relevant properties of the resultant associated to a homogeneous map $\Phi:\AA^{N+1}\to\AA^{N+1}$.  

\begin{prop}\label{ResProp}
Let $\Phi:\AA^{N+1}\to \AA^{N+1}$ be a map defined over a field $K$ by an $(N+1)$-tuple $\Phi=(\Phi_0,\dots,\Phi_N)$ of forms of some common degree $d\geq1$ in the variables $x_0,x_1,\dots,x_N$, and let $\Res(\Phi)$ denote the resultant of $\Phi$.
\begin{itemize}
	\item[(i)]  $\Res(\Phi)=0$ if and only if $\Phi(\a)=\0$ for some $\a\in\AA^{N+1}(\Kbar)\setminus\0$.
	\item[(ii)]  If $A\in\GL_{N+1}(K)$ is a linear automorphism of $\AA^{N+1}$ defined over $K$, then $\Res(A\circ\Phi\circ A^{-1})=\det(A)^{C(N,d)}\Res(\Phi)$ for some integer $C(N,d)$ depending only on $N$ and $d$.
\end{itemize}
\end{prop}
\begin{proof}
Part (i) is standard, see \cite{vanderWaerden}, $\S$82.  Part (ii) follows from \cite{MR1227515}, Cor. 5.
\end{proof}

\begin{proof}[Proof of Theorem~\ref{GMMTheoremIntro}]
Let $\Phi:\AA^{N+1}\to \AA^{N+1}$ be an arbitrary homogeneous lift for $\phi$.  For each $\pfrak\in M_R$, let $\Phi_\pfrak:\AA^{N+1}\to \AA^{N+1}$ be a minimal $\pfrak$-integral model for $\phi$; thus $\Phi_\pfrak= A_\pfrak\circ\Phi\circ A_\pfrak^{-1}$ for some $ A_\pfrak\in\GL_{N+1}(K)$.  If $S$ denotes the finite set of $\pfrak\in M_R$ for which some coefficient of $\Phi$ is not $R_\pfrak$-integral, or for which $\Res(\Phi)$ is not an $R_\pfrak$-unit, then we may take $\Phi_\pfrak=\Phi$ and $ A_\pfrak=I$ for all $\pfrak\not\in S$.

By Proposition~\ref{MainGroupThm}, there exists $ A\in\GL_{N+1}(K)$ such that $ A_\pfrak A^{-1}\in\GL_{N+1}(R_\pfrak)$ for each $\pfrak\in M_R$.  Consider the model $\Psi:\AA^{N+1}\to \AA^{N+1}$ for $\phi$ defined by $\Psi= A\circ\Phi\circ A^{-1}$.  For each $\pfrak\in M_R$, we have 
\begin{equation}\label{PsiConj}
\Psi=( A A_\pfrak^{-1})\circ\Phi_\pfrak\circ( A A_\pfrak^{-1})^{-1}.
\end{equation}
Since $A A_\pfrak^{-1}=(A_\pfrak A^{-1})^{-1}\in\GL_{N+1}(R_\pfrak)$ and $\Phi_\pfrak$ has coefficients in $R_\pfrak$, it follows from $(\ref{PsiConj})$ that $\Psi$ has coefficients in $R_\pfrak$ as well; since this holds for arbitrary $\pfrak\in M_R$, it follows from $(\ref{IntegralLocalGlobal})$ that $\Psi$ has coefficients in $R$.  Finally, since $\ord_\pfrak(\det( A A_\pfrak^{-1}))=0$, it follows from $(\ref{PsiConj})$ and Proposition~\ref{ResProp} that 
$$\ord_\pfrak(\Res(\Psi))=\ord_\pfrak(\Res(\Phi_\pfrak)),$$ 
and so $\Psi$ is $\pfrak$-minimal for each $\pfrak\in M_R$.  
\end{proof}

\begin{proof}[Proof of Corollary~\ref{EGRTheoremIntro}]
Since $\phi$ has everywhere good reduction, the model $\Psi$ constructed in Theorem~\ref{GMMTheoremIntro} satisfies $\ord_\pfrak(\Res(\Psi))=0$ for all nonzero prime ideals $\pfrak$ of $R$, and therefore $(\ref{IntegralLocalGlobal})$ implies that $\Res(\Psi)\in R^\times$.
\end{proof}

\end{document}